\title{\bf{Surfaces of general type with maximal  \\ Picard number near the Noether line  }}
\author{
	NGUYEN BIN and VICENTE LORENZO\\
}
\date{}

\newcommand{\BinAddresses}{{
		\bigskip
		\footnotesize
        \text{Nguyen Bin,}\par\nopagebreak	
		\text{Department of Mathematics and Statistics,}\par\nopagebreak	
		\text{Quy Nhon University,}\par\nopagebreak	
            \text{170 An Duong Vuong Street, Quy Nhon,}\par\nopagebreak
		\text{Vietnam.}\par\nopagebreak		
		\textit{E-mail address}: \texttt{nguyenbin@qnu.edu.vn}
		
}}

\newcommand{\VicenteAddresses}{{
		\bigskip
		\footnotesize
        \text{Vicente Lorenzo,}\par\nopagebreak	
		\text{Telematic Engineering Department,}\par\nopagebreak	
		\text{Universidad Carlos III de Madrid,}\par\nopagebreak	
            \text{Avenida de la Universidad 30, Leganés (Madrid),}\par\nopagebreak	
		\text{Spain.}\par\nopagebreak		
		\textit{E-mail address}: \texttt{vlorenzogarcia@gmail.com}
		
}}

\newcommand\blfootnote[1]{%
	\begingroup
	\renewcommand\thefootnote{}\footnote{#1}%
	\addtocounter{footnote}{-1}%
	\endgroup
}

\documentclass[10pt]{article}
\usepackage{amsmath, amsthm, amssymb, mathrsfs, amscd, amsthm, amscd,amsfonts,calligra,mathrsfs,adjustbox,lipsum}
\usepackage{indentfirst,url}
\usepackage[all]{xy}
\usepackage{url}
\usepackage{tikz}
\usepackage{tikz-cd}
\usepackage[colorlinks,plainpages]{hyperref}
\usepackage{paracol}
\usepackage{float}
\usetikzlibrary{arrows,decorations.markings}

\hypersetup{
	colorlinks=true,
	linkcolor=blue,
	filecolor=magenta,      
	urlcolor=cyan,
}
\DeclareMathOperator{\rank}{rank}

\makeatletter
\def\@endtheorem{\endtrivlist\@endpefalse } 
\makeatother

\makeatletter
\def\th@plainitalic{%
  \thm@notefont{}
  \itshape 
}
\makeatother

\theoremstyle{plainitalic}

\newtheorem{Theorem}{Theorem}
\newtheorem{Corollary}{Corollary}
\newtheorem{Proposition}{Proposition}
\newtheorem{Lemma}{Lemma}

\newtheorem{Remark}{Remark}
\newtheorem{Definition}{Definition}
\newtheorem*{Acknowledgments}{ACKNOWLEDGMENTS}

\newcommand{\MSC}{\textbf{Mathematics Subject Classification (2010):}}

\newcommand{\Key}{\textbf{Key words:}}

\begin{document}
\maketitle
\begin{abstract}
The first published non-trivial examples of algebraic surfaces of general type with maximal Picard number are due to Persson, who constructed surfaces with maximal Picard number on the Noether line $K^2=2\chi-6$ for every admissible pair $(K^2,\chi)$ such that $\chi \not\equiv 0 \text{ mod } 6$. In this note, given a non-negative integer $k$, algebraic surfaces of general type with maximal Picard number lying on the line $K^2=2\chi-6+k$ are constructed for every admissible pair $(K^2,\chi)$ such that  $\chi\geq 2k+10$. These constructions, obtained as bidouble covers of rational surfaces, not only allow to fill in Persson's gap on the Noether line, but they provide infinitely many new examples of algebraic surfaces of general type with maximal Picard number above the Noether line.
\end{abstract}

\blfootnote{\MSC{ 14J29}.}
\blfootnote{\Key{ Picard number, Surfaces of general type, Abelian covers.}}

\section{Introduction} \label{Introduction}

Let $X$ be a smooth projective surface over the complex numbers $\mathbb{C}$. The Neron-Severi group $NS(X)$ of $X$ is the  group of divisors of $X$ modulo numerical equivalence. The group $NS(X)$ is finitely generated and its rank $\rho(X)$ is called the Picard number of $X$. The Picard number of $X$ is bounded above by the Hodge number $h^{1,1}(X)=\text{dim } H^1(X,\Omega^1_X)$ and  $X$ is said to have maximal Picard number if $\rho(X)=h^{1,1}(X)$. Surfaces $X$ with geometric genus $p_g(X)=0$ always have maximal Picard number. Excepting this trivial case, examples of algebraic surfaces of general type with maximal Picard number are very scarce in the literature, being Beauville's paper \cite{MR3322784} a remarkable reference on the topic. 
It is convenient to give some context 
before stating our main results.\\

A pair of strictly positive integers $(K^2, \chi)$ is said to be an admissible pair if it satisfies 
Noether's inequality $K^2\geq 2\chi-6$ and the  Bogomolov–Miyaoka–Yau inequality $K^2\leq 9 \chi$. The relevance of this definition lies on the fact that the pair $(K^2_X, \chi(\mathcal{O}_X))$ consisting of the self-intersection of the canonical class $K^2_X$ and the holomorphic Euler characteristic $\chi(\mathcal{O}_X)$ of a minimal surface of general type $X$ is an admissible pair (cf. \cite[Chapter VII]{MR2030225}). Minimal surfaces of general type $X$ such that $K^2_X=2\chi(\mathcal{O}_X)-6$ (resp.  $K^2_X=2\chi(\mathcal{O}_X)-5$) are said to lie on the Noether line (resp. one above the Noether line)
and they are known as even (resp. odd) Horikawa surfaces because they were thoroughly studied by Horikawa \cite{MR0424831} (resp. \cite{MR0460340}). The first article including non-trivial examples of algebraic surfaces of general type with maximal Picard number is due to Persson \cite{MR661198}, who constructed, among other examples, infinitely many even Horikawa surfaces with maximal Picard number. More precisely, Persson \cite[Theorem 1]{MR661198} proved that given an admissible pair $(K^2,\chi)$ such that $K^2=2\chi-6$ and $\chi\not\equiv 0 \text{ mod } 6$, then all the connected components of $\mathfrak{M}_{K^2,\chi}$ contain canonical models with maximal Picard number, where $\mathfrak{M}_{K^2,\chi}$ is Gieseker's moduli space \cite{MR498596} of canonical models of surfaces of general type with self-intersection of the canonical class $K^2$ and holomorphic Euler characteristic $\chi$.  More surfaces of general type with maximal Picard number can be found in Beauville's paper \cite{MR3322784} and the references therein or \cite{2014arXiv1406.2143A}, \cite{MR3460339}, \cite{MR3683423}, \cite{2016arXiv161100470S}. It is worth noting that most of these examples have low geometric genus.\\

Persson obtained his examples of surfaces of general type with maximal Picard number using the same approach that he applied to tackle the geographical question \cite{MR631426}. That is to say, he considered double covers of rational surfaces whose branch locus has a very particular configuration. The fact that considering bidouble covers of rational surfaces instead of double covers of rational surfaces turned out to be very efficient to deal with the geographical question \cite{MR4626843} together with the fact that every Horikawa surface can be degenerated into   a bidouble cover of a rational surface \cite{MR4440715}, encouraged the authors to consider bidouble covers of rational surfaces in order to obtain new examples of surfaces of general type with maximal Picard number. The main result of this paper, whose proof relies on the technique just mentioned, is the following:

\begin{Theorem}\label{CorollaryFirstFam}
 	Given an integer $k\geq 0$ let $(K^2,\chi)$ be an admissible pair such that 
	$K^2=2\chi-6+k$. If $\chi\geq 2k+10$, then $\mathfrak{M}_{K^2,\chi}$ contains canonical models whose minimal resolution has maximal Picard number.
\end{Theorem}

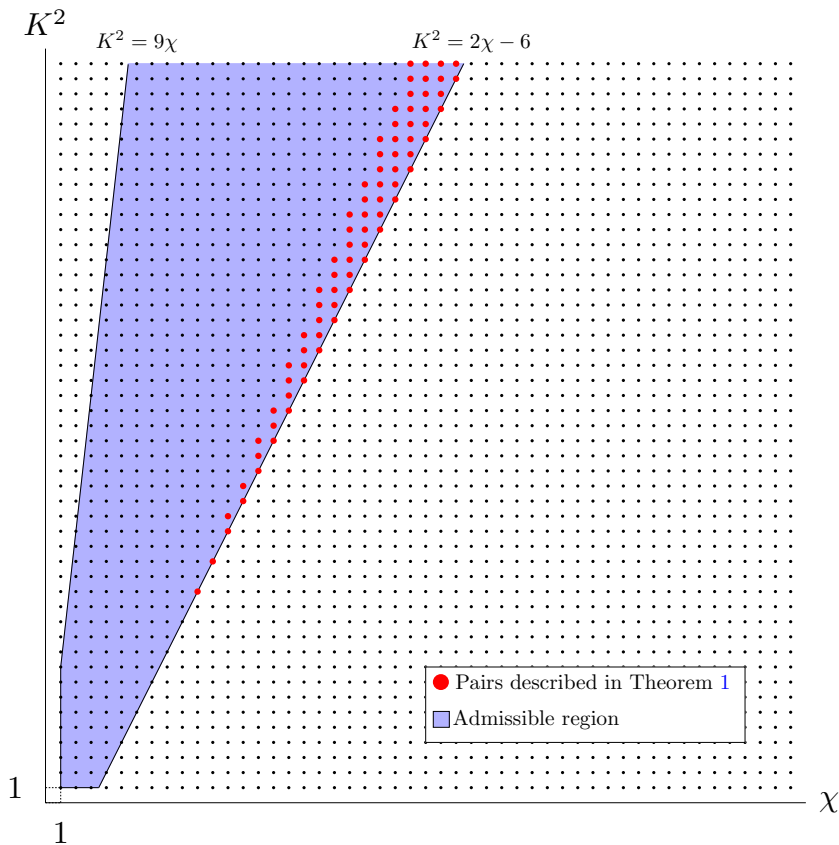
\begin{figure}[ht]
    \centering
    \scalebox{0.2}{
        \begin{tikzpicture}
    \draw[help lines, color=black, dashed, ultra thick] (0,0) grid (1,1) node [scale=6] at (1,-2) {$1$} node [scale=6] at (-2,1) {$1$};
                 \filldraw[blue!30, opacity=0.2] (5.44,49) -- (1,9) -- (1,1) -- (3.5,1) -- (27.5,49) -- cycle;
\foreach \x [count=\xi] in {1,...,49}
\foreach \y [count=\yi] in {1,...,49}
\fill   (\x,\yi) circle[radius=3pt];
            \draw[ultra thick] 
            (0,0)--(50,0) node[right, scale=6]{$\chi$};
            \draw[-,ultra thick] (0,0)--(0,50) node[above, scale=6]{$K^2$};
            \draw[domain=7/2:27.5, smooth, variable=\x] plot ({\x}, {2*\x-6}) node [scale=4] at (28, 50.5) {$K^2=2\chi-6$};
            \draw[domain=1:5.44, smooth, variable=\x] plot ({\x}, {9*\x}) node [scale=4] at (6, 50.5) {$K^2=9\chi$};
            \draw[domain=1:7/2, smooth, variable=\x] plot ({\x}, {1});
            \draw[domain=1:9, smooth, variable=\y] plot ({1}, {\y});

            \foreach \x [count=\xi] in {1,...,49}
            \foreach \y [count=\yi] in {1,...,49}
            {
                \foreach \k in {0,...,10} 
                {
                    \pgfmathtruncatemacro\condition{(\x >= 2*\k + 10) && (\y == 2*\x - 6 + \k)}
                    \ifnum\condition=1
                        \fill[red] (\x,\yi) circle[radius=6pt];
                    \fi
                }
            }
        \filldraw [white, opacity=1] (25,4) -- (25,9) -- (46,9) -- (46,4) -- cycle;
        \draw [thick,black] (25,9) rectangle (46, 4);
        
        \fill[red] (26,8) circle[radius=15pt];
        \node[scale=4] at (36,8) {Pairs described in Theorem \ref{CorollaryFirstFam}};
        
        \filldraw[blue!30, opacity=0.2, draw=black] (25.5,5) rectangle ++(30pt,30pt);
        \node[scale=4] at  (32.25,5.5) {Admissible region};
   
        \end{tikzpicture}
    }
    \caption{
        Plane $(K^2, \chi)$ where the  pairs described in Theorem \ref{CorollaryFirstFam} and the admissible region  (i.e. the region  delimited by the inequalities  $K^2\geq 1, \chi\geq 1, K^2\leq 9\chi, K^2\geq 2\chi-6$)   have been highlighted.
   }
    \label{fig:K2chiplane}
\end{figure}

As can be seen from Figure \ref{fig:K2chiplane} and will be elucidated in Remark \ref{ToArea} below,  the pairs described in Theorem \ref{CorollaryFirstFam} are precisely those admissible pairs $(K^2,\chi)$ such that $K^2\leq \frac{5}{2}\chi-11$. In particular, Remark \ref{density} will show that the set of quotients $\frac{K^2_X}{\chi(\mathcal{O}_X)}$ for some canonical model $X$ whose minimal resolution has maximal Picard number is dense in the interval $[2,\frac{5}{2}]$.

\medskip

Combining Persson and Horikawa's work with Theorem \ref{CorollaryFirstFam} we can get the following result regarding even Horikawa surfaces:

\begin{Theorem}\label{EvenHorikawaMaxPicard}
	Let $(K^2, \chi)$ be an admissible pair such that $K^2=2\chi-6$. Then all the connected components of $\mathfrak{M}_{K^2,\chi}$ contain canonical models whose minimal resolution has maximal Picard number.
\end{Theorem}

Combining Theorem \ref{CorollaryFirstFam} with some scattered examples we can get the following result regarding odd Horikawa surfaces:

\begin{Theorem}\label{OddHorikawaMaxPicard}
	Let $(K^2, \chi)$ be an admissible pair such that $K^2=2\chi-5$ and $\chi\neq 8$. Then $\mathfrak{M}_{K^2,\chi}$ contains canonical models whose minimal resolution has maximal Picard number.
\end{Theorem}

As it is now clear from the discussion above, Theorem \ref{CorollaryFirstFam} not only allows to fill in the missing case $\chi\equiv 0 \text{ mod } 6$ of \cite[Theorem 1]{MR661198}, but also provides infinitely many new examples of surfaces of general type with maximal Picard number lying one or more above the Noether line.\\

The paper is structured as follows. In Section \ref{BidoubleCovers}
it is explained how to construct bidouble covers and how to obtain information about them. Section \ref{BasicsPicardNumber} contains the basics on the Picard number that will be needed throughout the paper. In Section \ref{Proof of the main results} two families of surfaces with maximal Picard number are constructed. Theorem \ref{CorollaryFirstFam} will be an immediate consequence of these constructions. Finally, Section \ref{HorikawaSurfaces} is devoted to explain how Theorem \ref{EvenHorikawaMaxPicard} and Theorem \ref{OddHorikawaMaxPicard} can be derived  from the constructions of the previous section.\\

\textbf{Notation and conventions.} All varieties are assumed to be algebraic and  projective varieties over the complex numbers $\mathbb{C}$. Linear
equivalence of line bundles or divisors will be denoted by $\cong$, whereas $\equiv$ will  denote congruence modulo a given integer. The Hirzebruch surface $\mathbb{P}(\mathcal{O}_{\mathbb{P}^1}\oplus \mathcal{O}_{\mathbb{P}^1}(-e))$ will be denoted by $\mathbb{F}_e$ and, by abusing the notation, the negative section $\Delta_0$ and a general fiber $F$ of $\mathbb{F}_e$ will just mean two intersecting fibers if $e=0$. The rest of the notation is standard in
algebraic geometry.

	\section{Bidouble covers}
\label{BidoubleCovers}


A $\mathbb{Z}_2^2$-cover of a variety $Y$ is a finite map $f\colon X\to Y$ together with a faithful action of 
$\mathbb{Z}_2^2$ on $X$ such that $f$ expresses $Y$ as the quotient $X/\mathbb{Z}_2^2$.
Catanese \cite{MR755236} proved the structure theorem for smooth $\mathbb{Z}_2^2$-covers.  
By {\cite[Section 2]{MR755236}}
or {\cite[Theorem 2.1]{MR1103912}},
to define a normal $\mathbb{Z}_2^2$-cover $X\to Y$ of a
smooth and irreducible projective variety $Y$, it suffices to consider both:
\begin{itemize}
	\item Effective divisors $B_1, B_2, B_3$ on $Y$ such that the branch locus $B=B_1+B_2+B_3$ is reduced.
	\item Non-trivial line bundles $L_1, L_2, L_3$ on $Y$ satisfying $2L_1\cong B_2+B_3, 2L_2\cong B_1+B_3$ and such that  $L_3\cong L_1+L_2-B_3$.
\end{itemize}
The set $\{L_i,B_j\}_{i,j}$ is known as the building data of the cover.

\begin{Proposition}[{{\cite[Section 2]{MR755236}}} or {{\cite[Proposition 4.2]{MR1103912}}}]\label{Z22Formulas}
	Let $Y$ be a smooth surface and $f\colon X\to Y$ a smooth $\mathbb{Z}_2^2$-cover with building data 
	$\{L_i,B_j\}_{i,j}$. Then:
	\begin{equation*}
		\begin{split}
			2K_X & \cong f^*(2K_Y+B_1+B_2+B_3),\\
			K_X^2 & =(2K_Y+B_1+B_2+B_3)^2,\\
			p_g(X) & =p_g(Y)+\sum_{i=1}^3h^0(K_Y+L_i),\\
			\chi(\mathcal{O}_X) & =4\chi(\mathcal{O}_Y)+\frac{1}{2}\sum_{i=1}^3L_i(L_i+K_Y).
		\end{split}
	\end{equation*}
\end{Proposition}

We notice that Proposition \ref{Z22Formulas} holds in the case $X$ has $ADE$ singularities (see \cite{MR2030225} for  information about ADE singularities on curves and surfaces). For a more general statement, we direct the reader to the work of Bauer and Pignatelli \cite[Section 2]{MR4278662}.


    \begin{Remark}\label{LocalCoord}
	Let $f\colon X\to Y$ be a $\mathbb{Z}_2^2$-cover of a smooth surface $Y$ with  building data $\{L_i,B_j\}_{i,j}$ and let $Q=f(P)$ be an intersection point of $B_1$ and $B_2$ that is not contained in $B_3$. Let us denote by $b_j$ a local equation for $B_j$ around $Q=f(P)\in Y$ and let $w_i$ be a local coordinate on $L_i$ 
	for every $i,j\in\{1,2,3\}$. According to \cite[Proposition 2.3]{MR755236}, the following equations define $X$ locally around $P$:
	$$  w_1^2=b_2b_3, \qquad w_2^2=b_1b_3, \qquad w_3^2=b_1b_2, \qquad w_1w_2=w_3b_3, \qquad w_2w_3=b_1w_1, \qquad w_3w_1=b_2w_2.$$  
 Moreover, because $Q$ is not contained in $B_3$, we can assume $b_3 =1$.
\end{Remark}

We are going to use Remark \ref{LocalCoord} and its notation to prove the following:

\begin{Corollary}\label{Z22Sing}
	Let $f\colon X\to Y$ be a $\mathbb{Z}_2^2$-cover of a smooth surface $Y$ with  building data $\{L_i,B_j\}_{i,j}$ and let $Q=f(P)$ be an intersection point of $B_1$ and $B_2$ that is not contained in $B_3$.
	\begin{itemize}
		\item[i)] If both $B_1$ and $B_2$ are smooth at $Q$ but they intersect in such a way that $B_1+B_2$ has a singularity of type $A_{2n+1}$ at $Q$, then  $X$ has a singularity of type $A_n$ at $P$.
		\item[ii)] If $B_1$ has a singularity of type $A_1$ at $Q$, $B_2$ is smooth at $Q$ and they intersect in such a way that $B_1+B_2$ has a singularity of type $D_{2n+4}$ at $Q$, then  $X$ has a singularity of type $D_{n+3}$ at $P$.
		\item[iii)] If $B_1$ has a singularity of type $A_{n}$ at $Q$, $B_2$ is smooth at $Q$ and they intersect in such a way that $B_1+B_2$ has a singularity of type $D_{n+3}$ at $Q$, then  $X$ has a singularity of type $A_{2n+1}$ at $P$.
	\end{itemize}
\end{Corollary}

\begin{proof}
	In case i), since we are working over the complex numbers and $Q$ is not contained in $B_3$, we can assume that:
	$$b_1= x - y^{n+1},\qquad b_2= x + y^{n+1},\qquad b_3= 1.$$
	According to Remark \ref{LocalCoord} the following equations define $X$ locally around $P$:
	\begin{equation*}
	    \begin{split}
	        w_1^2 =x + y^{n+1}, \qquad w_2^2=x - y^{n+1}, \qquad w_3^2=x^2-y^{2n+2},\\  w_1w_2=w_3, \qquad w_2w_3=(x - y^{n+1})w_1, \qquad w_3w_1=(x + y^{n+1})w_2.
	    \end{split}
	\end{equation*} 
	Simplifying and eliminating redundant equations, we conclude that $X$ is defined locally around $P$ by:
         $$w_1^2=w_2^2+ 2y^{n+1}, $$
	which is the equation of a singularity of type $A_n$.\\

	\medskip
	
	In case ii) we can assume that:
	$$b_1= y(x - y^{n+1}),\qquad b_2= x + y^{n+1},\qquad b_3= 1.$$
	According to Remark \ref{LocalCoord} the following equations define $X$ locally around $P$:
	\begin{equation*}
	    \begin{split}
	        w_1^2=x + y^{n+1}, \qquad w_2^2=y(x - y^{n+1}), \qquad w_3^2=y(x^2-y^{2n+2}),\\
         w_1w_2=w_3, \qquad w_2w_3=y(x - y^{n+1})w_1, \qquad w_3w_1=(x + y^{n+1})w_2.
	    \end{split}
	\end{equation*} 
	Simplifying and eliminating redundant equations, we conclude that $X$ is defined locally around $P$ by:
 $$w_2^2=y(w_1^2- 2y^{n+1}),$$
	which is the equation of a singularity of type $D_{n+3}$.
	
	\medskip
	
	In case iii) we can assume that:
	$$b_1= x^2+y^{n+1},\qquad b_2= y,\qquad b_3= 1.$$
	According to Remark \ref{LocalCoord} the following equations define $X$ locally around $P$:
    \begin{equation*}
	    \begin{split}
	w_1^2= y, \qquad w_2^2=x^2+y^{n+1}, \qquad w_3^2=y(x^2+y^{n+1}),\\
         w_1w_2=w_3, \qquad w_2w_3=(x^2+y^{n+1})w_1, \qquad w_3w_1=yw_2.
	    \end{split}
	\end{equation*}
	Simplifying and eliminating redundant equations, we conclude that $X$ is defined locally around $P$ by:
$$w_2^2=x^2+w_1^{2n+2},$$
	which is the equation of a singularity of type $A_{2n+1}$.
\end{proof}

\begin{Remark}
	It follows from the proof of Corollary \ref{Z22Sing}, i) that 
	if both $B_1$ and $B_2$ are smooth at $Q$, they intersect in such a way that $B_1+B_2$ has a singularity of type $A_{1}$ at $Q$ and $Q$ is not contained in $B_3$, then  $X$ is smooth at $P$.
\end{Remark}

\begin{Remark}\label{DoubleSing}
	It can be proved as in Corollary \ref{Z22Sing} that an ADE singularity of $B_i,i\in\{1,2,3\}$ that is disjoint from the two other divisors of the branch locus induces two singularities of the same type on $X$.
\end{Remark}

For more information about the singularities arising from bidouble covers we address the reader to \cite{MR879190}. See also \cite{MR1103912}.

	\section{Basics on the Picard number}\label{BasicsPicardNumber}

First of all, since $h^{1,1}(X)$ is an upper bound on the Picard number $\rho(X)$ of a smooth projective surface $X$, we start this section with a formula that
will be used implicitly several times throughout this paper (cf. \cite{MR3439840}):
\begin{equation*}\label{h11}
	h^{1,1}(X)=10\chi(\mathcal{O}_X)-K^2_X-2q(X).
\end{equation*}

Now, from the definition of the Picard number of a smooth projective surface given in Section \ref{Introduction} we get the following:

\begin{Remark}\label{RankBunch}
	Let $M$ be the intersection matrix of a finite set of divisors on the smooth projective surface $X$. Then:
	\begin{equation*}
		\rho(X)\geq \rank(M).
	\end{equation*}    
\end{Remark}

Although it is straightforward from 
Remark \ref{RankBunch},  the following result will be fundamental in order to prove Theorem \ref{CorollaryFirstFam}.

\begin{Lemma}
    \label{PicardSurjectiveMorphism}
	Let $X$ be a canonical model with minimal resolution $\hat{X}\to X$ and whose singular set consists of:
	\begin{itemize}
		\item $\alpha_i$ singularities of type $A_i, i\in\mathbb{Z}_{\geq 1}$,
		\item $\beta_j$ singularities of type $D_j, j\in\mathbb{Z}_{\geq 4}$,
		\item $\gamma_k$ singularities of type $E_k, k\in\{6,7,8\}$.
	\end{itemize}
	Let us suppose that $X$ admits a surjective morphism $\pi:X\to Y$ onto a smooth and projective surface $Y$ and there exist numerically independent divisors $C_1,\ldots, C_n$ on $Y$.
	Then:
	\begin{equation*}
		\rho(\hat{X})\geq \sum_{i} i\cdot\alpha_i + 
		\sum_{j} j\cdot\beta_j+\sum_{k} k\cdot\gamma_k+n.
	\end{equation*}
\end{Lemma}

\begin{Remark}
	Throughout this note we will just apply Lemma \ref{PicardSurjectiveMorphism} under the following circumstances:
	\begin{itemize}
		\item $\pi\colon X\to \mathbb{P}^2$   is a $\mathbb{Z}_2^2$-cover of the projective plane $\mathbb{P}^2$, $n=1$ and $C_1$ is a line.
		\item $\pi\colon X\to \mathbb{F}_e$ is a $\mathbb{Z}_2^2$-cover of the Hirzebruch surface $\mathbb{F}_e$ for some integer $e \geq 0$, $n=2$ and the divisors $C_1,C_2$ are the negative section and a fiber.
	\end{itemize}
\end{Remark} 

We finish this section with an adaptation of a result of Persson that, although not strictly related to Picard numbers, will be crucial in order to construct surfaces with maximal Picard number.

\begin{Lemma}[{\cite[Lemma 4.1]{MR661198}}]\label{PerssonsTrick}
	Given a Hirzebruch surface $\mathbb{F}_e$ with two disjoint fibers $F_1,F_2$ there is for
	each integer $d$ a unique $d$-cyclic cover $\pi_d\colon \mathbb{F}_{de} \to \mathbb{F}_e$ branched at $F_1 $ and $ F_2$. Furthermore:
	\begin{itemize}
		\item[a)] If $C\cong \mathcal{O}_{\mathbb{F}_e}(a\Delta_0+b F)$ is an effective divisor not having $F_1$ and $F_2$ as components then $C_d:=\pi_d^*C\cong \mathcal{O}_{\mathbb{F}_{de}}(a\Delta_0+dbF)$.
		\item[b)] If $C$ has at most ADE singularities, then $C_d$ has at most ADE singularities if and only if the only singularities of
		$C$ lying on $F_1$ or $F_2$ are of type $A_n$ and they are transversal to these fibers.
		\item[c)] A singularity of type $A_{2n-1}$ of $C$ on $F_i,i\in\{1,2\}$ that is transversal to this fiber gives rise to a singularity of type $A_{2dn-1}$ of $C_d$. 
		\item[d)] A singularity of $C$ not lying on $F_1\cup F_2$ gives rise to $d$ singularities of the same type on $C_d$.
	\end{itemize}
\end{Lemma}

	\section{Two families of surfaces with maximal Picard number}\label{Proof of the main results}

This section is devoted to construct two families of surfaces of general type with maximal Picard number that will allow to prove Theorem \ref{CorollaryFirstFam}, Theorem \ref{EvenHorikawaMaxPicard} and Theorem \ref{OddHorikawaMaxPicard}. As we will see in Remark \ref{ComparingFamilies}, the contribution of the second family to the proof of these results is very small.
Nevertheless, the authors decided to include the whole family due to the scarcity of examples of surfaces of general type with maximal Picard number in the literature.\\

The main idea behind these constructions goes as follows. First, a highly singular configuration of curves $\{C_{\alpha}\}_{\alpha}$ on $\mathbb{P}^2$  is considered. Second, a point of $\mathbb{P}^2$ is blown-up so that $\{C_{\alpha}\}_{\alpha}$ induces a highly singular configuration of curves  $\{\overline{C}_{\alpha}\}_{\alpha}$ on $\mathbb{F}_1$. Next, Lemma \ref{PerssonsTrick} is applied. Using its notation, let us denote by $\{\tilde{C}_{\alpha}\}_{\alpha}$ the configuration of curves on $\mathbb{F}_d$ induced by $\{\overline{C}_{\alpha}\}_{\alpha}$ via $\pi_d$. Then, the singularities of type $A_{2n-1}$ of $\sum_{\alpha}\overline{C}_{\alpha}$ lying on $F_1\cup F_2$ become singularities of type $A_{2dn-1}$ on $\sum_{\alpha}\tilde{C}_{\alpha}$ and the singularities of $\sum_{\alpha}\overline{C}_{\alpha}$ not lying on $F_1\cup F_2$ are multiplied by $d$ on $\sum_{\alpha}\tilde{C}_{\alpha}$. Finally, fibers $\{F_{\beta}\}_{\beta}$ of $\mathbb{F}_d$ passing through singularities of type $A_{m}$ of $\sum_{\alpha}\tilde{C}_{\alpha}$ are chosen so that $\sum_{\alpha}\tilde{C}_{\alpha}+\sum_{\beta} F_{\beta}$ has many 
singularities of type $D_{m+3}$. Canonical models $X$  whose minimal resolution has maximal Picard number will be obtained as  $\mathbb{Z}_2^2$-covers $\varphi\colon X\to \mathbb{F}_d$ whose branch locus is made of the curves $\{\tilde{C}_{\alpha}\}_{\alpha} \cup \{ F_{\beta}\}_{\beta}$.

\subsection{A first family of surfaces with maximal Picard number}

The first family of surfaces of general type with maximal Picard number is described in the following:

\begin{Theorem}
    \label{FirstFam}
	Let $n,m,k$ be non-negative integers such that $(n,m,k)\neq (1,0,0), m\leq 2n+k\neq 0$ and  $m\equiv k (2)$ and denote $K^2=6n+2m+5k-4,
	\chi =3n+m+2k+1$. Then $\mathfrak{M}_{K^2,\chi}$ contains canonical models  whose minimal resolution has maximal Picard number.
\end{Theorem}

Before proving Theorem \ref{FirstFam}, we are going to show how Theorem \ref{CorollaryFirstFam} can be derived from it.

\begin{proof}[Proof of Theorem \ref{CorollaryFirstFam}]
	Since $\chi\geq 2k+10$ we can write
	$\chi-(2k+1)=3a+b$
	for some integers $a\geq 3, b\in\{0,1,2\}$. If $b\equiv k (2)$ we set $n=a, m=b$. If $b\not\equiv k (2)$ we set $n=a-1, m=b+3$. The result follows from Theorem \ref{FirstFam}.
\end{proof}

Let us see the proof of Theorem \ref{FirstFam} now.

\begin{proof}[Proof of Theorem \ref{FirstFam}]
	
	Let $ l_1, l_2, l_3, l_4 \in|\mathcal{O}_{\mathbb{P}^2}(1)|$ be four lines in general position and denote:
	\begin{equation*}
		P_1=l_1\cap l_2,\qquad
		P_2=l_2\cap l_3,\qquad
		P_3=l_1\cap l_3,\qquad
		P_4=l_1\cap l_4,\qquad
		P_5=l_2\cap l_4,\qquad
		P_6=l_3\cap l_4.
	\end{equation*}

	Let us also consider:
	\begin{equation*}
		l_5:=\langle P_1,P_6\rangle ,\qquad l_6:=\langle  P_2, P_4 \rangle ,\qquad Q:=l_5 \cap l_6,\qquad l_7:=\langle  Q,P_5\rangle , \qquad l_8:=\langle Q, P_3\rangle , \qquad P_7:=l_2 \cap l_8.
	\end{equation*}
	
	\begin{center}

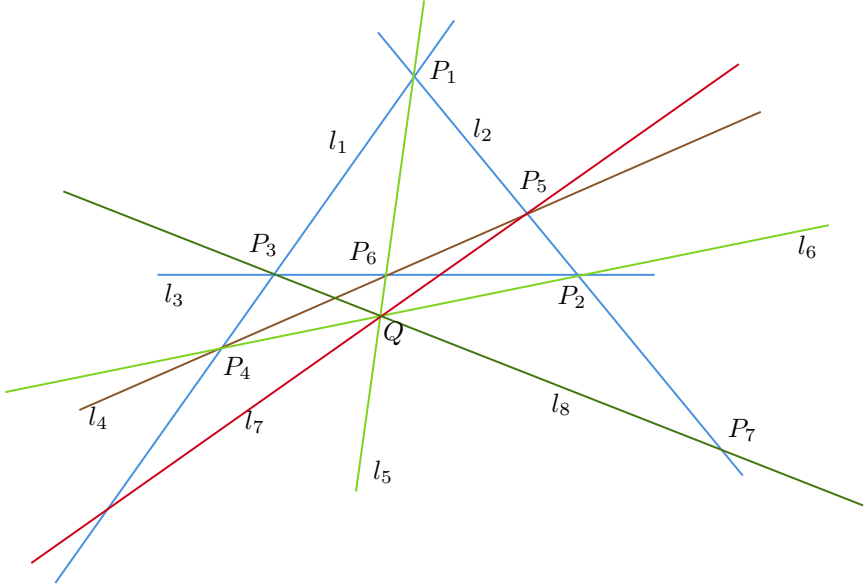
		\tikzset{every picture/.style={line width=0.75pt}} 
		
		\begin{tikzpicture}[x=0.75pt,y=0.75pt,yscale=-1,xscale=1]
			
			\draw [color={rgb, 255:red, 74; green, 144; blue, 226 }  ,draw opacity=1 ]   (154,157) -- (402,157) ;
			\draw [color={rgb, 255:red, 74; green, 144; blue, 226 }  ,draw opacity=1 ]   (103,312) -- (302,29) ;
			\draw [color={rgb, 255:red, 74; green, 144; blue, 226 }  ,draw opacity=1 ]   (264,35) -- (446,258) ;
			\draw [color={rgb, 255:red, 139; green, 87; blue, 42 }  ,draw opacity=1 ]   (115,225) -- (455,75) ;
			\draw [color={rgb, 255:red, 126; green, 211; blue, 33 }  ,draw opacity=1 ]   (287,19) -- (253,266) ;
			\draw [color={rgb, 255:red, 126; green, 211; blue, 33 }  ,draw opacity=1 ]   (78,216) -- (489,132) ;
			\draw [color={rgb, 255:red, 208; green, 2; blue, 27 }  ,draw opacity=1 ]   (444,51) -- (91,302) ;
			\draw [color={rgb, 255:red, 68; green, 119; blue, 15 }  ,draw opacity=1 ]   (107,115) -- (506,273) ;
			
			\draw (265,179) node [anchor=north west][inner sep=0.75pt]   [align=left] {$ Q $};
			\draw (288,48) node [anchor=north west][inner sep=0.75pt]   [align=left] {$ P_1 $};
			\draw (352,161) node [anchor=north west][inner sep=0.75pt]   [align=left] {$ P_2 $};
			\draw (198,135) node [anchor=north west][inner sep=0.75pt]   [align=left] {$ P_3 $};
			\draw (185,197) node [anchor=north west][inner sep=0.75pt]   [align=left] {$ P_4 $};
			\draw (437,228) node [anchor=north west][inner sep=0.75pt]   [align=left] {$ P_7 $};
			\draw (248,139) node [anchor=north west][inner sep=0.75pt]   [align=left] {$ P_6 $};
			\draw (333,103) node [anchor=north west][inner sep=0.75pt]   [align=left] {$ P_5 $};
			\draw (238,83) node [anchor=north west][inner sep=0.75pt]   [align=left] {$ l_1 $};
			\draw (310,77) node [anchor=north west][inner sep=0.75pt]   [align=left] {$ l_2 $};
			\draw (156,160) node [anchor=north west][inner sep=0.75pt]   [align=left] {$ l_3 $};
			\draw (118,222) node [anchor=north west][inner sep=0.75pt]   [align=left] {$ l_4 $};
			\draw (260,249) node [anchor=north west][inner sep=0.75pt]   [align=left] {$ l_5 $};
			\draw (472,136) node [anchor=north west][inner sep=0.75pt]   [align=left] {$ l_6 $};
			\draw (196,224) node [anchor=north west][inner sep=0.75pt]   [align=left] {$ l_7 $};
			\draw (349,215) node [anchor=north west][inner sep=0.75pt]   [align=left] {$ l_8 $};

		\end{tikzpicture}

	\end{center}

	Let $ \pi\colon \mathbb{F}_1 \to \mathbb{P}^2 $ be the blow-up of $ \mathbb{P}^2 $ at $ Q $. We denote by:
	\begin{center}
		\begin{tabular}{l l}
			$ \overline{l}_1, \overline{l}_2, \overline{l}_3, \overline{l}_4 $: & the pull-back of  $ l_1, l_2, l_3, l_4$, respectively,\\
			$ l'_5, l'_6, l'_7, l'_8 $:&the strict transforms of $ l_5 $, $ l_6 $, $ l_7 $, $ l_8 $, respectively.
		\end{tabular}	
	\end{center}

    Note that 
    $\overline{l}_i\cong \mathcal{O}_{\mathbb{F}_1}(\Delta_0+F)$ for all $i\in\{1,2,3,4\}$ and  $l'_i\cong \mathcal{O}_{\mathbb{F}_1}(F)$ for all $i\in\{5,6,7,8\}$.\\
    

	Let $\psi\colon \mathbb{F}_{2n+k} \to \mathbb{F}_1  $ be the $ (2n+k) $-cyclic cover branched at $ l'_5 + l'_6 $ (see Lemma \ref{PerssonsTrick}).\\

  Firstly, since $\overline{l}_i\cong \mathcal{O}_{\mathbb{F}_1}(\Delta_0+F)$  is an effective divisors not having $l'_5 $ and $ l'_6$ as components for all $i\in\{1,2,3,4\}$,  then
    $\psi^{*}\left( \overline{l}_i\right) \in |\mathcal{O}_{\mathbb{F}_{2n+k}}(\Delta_0 + (2n+k)F)|$  for all $i\in\{1,2,3,4\}$
   by Lemma \ref{PerssonsTrick}, where $\psi^*$ is the pull-back of $\psi$. Set $\tilde{l}_i :=\psi^{*}\left( \overline{l}_i\right) \in |\mathcal{O}_{\mathbb{F}_{2n+k}}(\Delta_0 + (2n+k)F)|$ for all $i \in \{ 1,2,3,4\}$.\\

    Secondly, being $\psi$   branched at $ l'_5 + l'_6 $, we obtain that    
    \begin{equation*}
		\begin{split}			
			\psi^{*}\left( l'_5\right) & =(2n+k)\tilde{l}_5,\\
			\psi^{*}\left( l'_6\right) & =(2n+k)\tilde{l}_6,
		\end{split}
	\end{equation*}
    
    for some fibers $\tilde{l}_5, \tilde{l}_6 \in |\mathcal{O}_{\mathbb{F}_{2n+k}}(F)|$.\\

    Lastly, because the divisors $ l'_7$ and $ l'_8 $ do not meet the branch locus $ l'_5 + l'_6$ of $\psi$, we get
    \begin{equation*}
		\begin{split}			
			\psi^{*}\left( l'_7\right) & 
			=\tilde{l}^{1}_7+\tilde{l}^{2}_7+\cdots + \tilde{l}^{2n+k}_7,\\
			\psi^{*}\left( l'_8\right) & 
			=\tilde{l}^{1}_8+\tilde{l}^{2}_8+\cdots + \tilde{l}^{2n+k}_8,
		\end{split}
	\end{equation*}
    
    for some fibers $\tilde{l}^{1}_7,\tilde{l}^{2}_7,\cdots, \tilde{l}^{2n+k}_7, \tilde{l}^{1}_8, \tilde{l}^{2}_8, \cdots, \tilde{l}^{2n+k}_8 \in |\mathcal{O}_{\mathbb{F}_{2n+k}}(F)|$.\\

	We consider the following divisors on $ \mathbb{F}_{2n+k} $:
	\begin{equation*}
		\begin{split}
			& B_{1} := \Delta_0\\
			& B_{2}: =\begin{cases} 
				\tilde{l}_4 \text{ if } k=0, \\
				\tilde{l}_4+ \tilde{l}^{1}_7+ \cdots + \tilde{l}^{k}_7 \text{ otherwise.}\end{cases} \\ 
			& B_{3} : = \begin{cases} 
				\tilde{l}_1 + \tilde{l}_2 + \tilde{l}_3 + \tilde{l}_5 + \tilde{l}_6 \text{ if } m=0, \\
				\tilde{l}_1 + \tilde{l}_2 + \tilde{l}_3 + \tilde{l}_5 + \tilde{l}_6+\tilde{l}^{1}_8+\cdots + \tilde{l}^{m}_8  \text{ otherwise.}
			\end{cases}
		\end{split}
	\end{equation*}
	
	Note that $B_1\in |\mathcal{O}_{\mathbb{F}_{2n+k}}(\Delta_0)|, B_2 \in |\mathcal{O}_{\mathbb{F}_{2n+k}}(\Delta_0 + (2n+2k)F)|, B_3 \in |\mathcal{O}_{\mathbb{F}_{2n+k}}(3\Delta_0 + \left( 6n+m+3k+2\right)F)| $. 
	We also consider the following line bundles on $ \mathbb{F}_{2n+k} $:
	\begin{equation*}
		\begin{split}
			& L_1:= \mathcal{O}_{\mathbb{F}_{2n+k}}\left(2\Delta_0+\left(4n+\frac{m+5k}{2}+1\right)F\right), \\
			& L_2:=\mathcal{O}_{\mathbb{F}_{2n+k}}\left(2\Delta_0+\left(3n+\frac{m+3k}{2}+1\right)F\right),\\
			& L_3:= \mathcal{O}_{\mathbb{F}_{2n+k}}\left(\Delta_0+(n+k)F\right).
		\end{split}
	\end{equation*}

	The building data $\{L_i,B_j\}_{i,j}$ defines a $ \mathbb{Z}^2_2 $-cover $ \varphi\colon X \to \mathbb{F}_{2n+k}$.
	On the one hand, taking into account Lemma \ref{PerssonsTrick}, the singular set of $B_3$ consists of:
	\begin{itemize}
		\item One singularity of type $ D_{4n+2k+2} $ coming from $ P_1 $ (note that the singularity of type $A_1$ of $\overline{l}_1 + \overline{l}_2$ induces a singularity of type $A_{4n+2k-1}$ on $\tilde{l}_1+\tilde{l}_2$ by Lemma \ref{PerssonsTrick} and it becomes a singularity of type $D_{4n+2k+2}$ if we add $\tilde{l}_5$).
		\item One singularity of type $ D_{4n+2k+2} $ coming from $ P_2 $ (the previous reasoning applies verbatim if we exchange $l_1,l_2,l_5$ by $l_2,l_3,l_6$ respectively).
		\item $m$ singularities of type $ D_{4} $ coming from $ P_3 $.
		\item $ 2n+k-m $ singularities of type $ A_1 $ coming from $ P_3 $.
		\item $m$ singularities of type $ A_{1} $ coming from $ P_7 $.
	\end{itemize}
	The singular set of $B_2+B_3$ restricted to $ B_2\cap B_3 $ consists of:
	\begin{itemize}
		\item One singularity of type $ D_{4n+2k+2} $ coming from $ P_4 $.
		\item One singularity of type $ D_{4n+2k+2} $ coming from $ P_6 $.
		\item $k$ singularities of type $ D_{4} $ coming from $ P_5 $.
	\end{itemize}
	
	Thus, the singular set of $X$ consists of: 
	\begin{itemize}
		\item $ 4 $ singularities of type $ D_{4n+2k+2} $ by Remark \ref{DoubleSing}.
		\item $ 2m $ singularities of type $ D_{4} $ by Remark \ref{DoubleSing}.
		\item $ 4n+2k$ singularities of type $ A_1 $ by Remark \ref{DoubleSing}.
		\item $ 2 $ singularities of type $ D_{2n+k+2} $ by Corollary \ref{Z22Sing},$ii)$.
		\item $k$ singularities of type $ A_3 $ by Corollary \ref{Z22Sing},$iii)$.
	\end{itemize}

	On the other hand, by Proposition \ref{Z22Formulas} we have that $K_X$ is ample because $2K_{X}$ is the pull-back via $\varphi$ of the ample divisor $\mathcal{O}_{\mathbb{F}_{2n+k}}\left(\Delta_0 + \left( 4n+m+3k-2\right)F\right)$. Moreover:
	\begin{equation*}
		\begin{split}
			K_{X}^2 & = \left( \Delta_0 + \left(4n+m+3k-2\right)F\right)^2=6n+2m+5k-4,\\
			\chi\left( \mathcal{O}_{X}\right) &=3n+m+2k+1,\\
			p_g\left( X\right) & =3n+m+2k,\\
			q\left( X\right)  & = 0,\\
			h^{1,1}\left( X\right) & =24n+8m+15k+14.
		\end{split}
	\end{equation*}	
	We conclude that $X\in\mathfrak{M}_{K^2,\chi}$ is a canonical model whose minimal resolution has 
	maximal Picard number $  h^{1,1}\left( X\right) = 24n+8m+15k+14$
	by Lemma \ref{PicardSurjectiveMorphism}.
\end{proof}

\begin{Remark} \label{ExtraCases1Fam}
	The construction considered in the proof of Theorem \ref{FirstFam} works for every admissible pair $(K^2,\chi)$ satisfying the conditions stated on Theorem \ref{CorollaryFirstFam}, but it also works in some other cases. For instance:
	\begin{itemize}
		\item When $K^2=2\chi-6$, this construction also produces canonical models $X$  in $\mathfrak{M}_{K^2,\chi}$  whose minimal resolution has maximal Picard number for the following pairs:
		\begin{center}
			\begin{tabular}{ |c||c|c|c|} 
				\hline
				$K^2$ &  $6$ &$8$&$12$ \\ 
				\hline 
				$\chi$ &  $6$ &$7$&$9$\\ 
				\hline
			\end{tabular}
		\end{center}
  Indeed, using the notation of Theorem \ref{FirstFam}, $(K^2,\chi)=(6,6)$ can be obtained setting $(n,m,k)=(1,2,0)$, $(K^2,\chi)=(8,7)$ can be obtained setting $(n,m,k)=(2,0,0)$ and $(K^2,\chi)=(12,9)$ can be obtained setting $(n,m,k)=(2,2,0)$. 
		\item When $K^2=2\chi-5$, this construction also produces canonical models $X$  in $\mathfrak{M}_{K^2,\chi}$  whose minimal resolution has maximal Picard number for the following pairs:
		\begin{center}
			\begin{tabular}{ |c||c|c|c|c|} 
				\hline
				$K^2$ & $3$ & $9$& $13$& $15$\\ 
				\hline 
				$\chi$ & $4$ & $7$& $9$& $10$\\ 
				\hline
			\end{tabular}
		\end{center}
  Indeed, using the notation of Theorem \ref{FirstFam}, $(K^2,\chi)=(3,4)$ can be obtained setting $(n,m,k)=(0,1,1)$, $(K^2,\chi)=(9,7)$ can be obtained setting $(n,m,k)=(1,1,1)$, $(K^2,\chi)=(13,9)$ can be obtained setting $(n,m,k)=(1,3,1)$ and $(K^2,\chi)=(15,10)$ can be obtained setting $(n,m,k)=(2,1,1)$. 
	\end{itemize}     
\end{Remark}

\begin{Remark}\label{ToArea}
On the one hand, let  $(K^2,\chi)$ be an  admissible pair
such that $K^2=2\chi-6+k$ and $\chi\geq 2k+10$ for a given integer $k\geq 0$. Then
\begin{equation*}
    K^2=2\chi-6+k\geq 2\chi-6
\end{equation*}
and
\begin{equation*}
    K^2=2\chi-6+k\leq 2\chi-6+\frac{\chi-10}{2}= \frac{5}{2}\chi-11.
\end{equation*}
	Hence,  admissible pairs $(K^2,\chi)$ satisfying the conditions stated on Theorem \ref{CorollaryFirstFam} lie on the region $2\chi -6 \leq K^2\leq \frac{5}{2}\chi-11$.
 
  On the other hand, given an admissible pair $(K^2, \chi)$ such that  $K^2\leq \frac{5}{2}\chi-11$ and denoting $k=K^2-(2\chi-6)$, we have that $K^2=2\chi-6+k$ and $\chi\geq 2k+10$ because:
 \begin{equation*}
     k=K^2-(2\chi-6)\leq \frac{5}{2}\chi-11-(2\chi-6)=\frac{\chi-10}{2}.
 \end{equation*}

 We conclude that the set of pairs describe in Theorem \ref{CorollaryFirstFam} is precisely the set of admissible pairs $(K^2,\chi)$ such that $K^2\leq \frac{5}{2}\chi-11$ (see Figure \ref{fig:K2chiplane}).
\end{Remark}

\begin{Remark}\label{density}
    It follows from Theorem \ref{CorollaryFirstFam} that given a rational number $q\in \left(2,\frac{5}{2}\right)$ there exists a canonical model $X$ whose minimal resolution has maximal Picard number and such that the quotient 
    $\frac{K^2_X}{\chi(\mathcal{O}_X)}$ is equal to $q$. In particular, the set of such quotients is dense in the interval $\left[2,\frac{5}{2}\right]$ by  the density of 
    $\mathbb{Q}$ in 
    $\mathbb{R}$.
    
    Indeed, given $q\in \left(2,\frac{5}{2}\right)\cap \mathbb{Q}$ let us consider $\frac{a}{b}:=q-2\in\left(0,\frac{1}{2}\right)$, where $a,b$ are positive integers. Then $\delta:=(b-2a)>0$ and there exists a positive integer $\lambda$ such that $\delta\lambda\geq 22$. We set $\chi=b\lambda , K^2=(a+2b)\lambda$ and we claim that $K^2\leq \frac{5}{2}\chi - 11$. In fact, this inequality is equivalent to $(a+2b)\lambda\leq \frac{5}{2} b \lambda - 11$, which is equivalent to  $2a\lambda+4b\lambda\leq 5b\lambda -22$, which holds true if and only if $22\leq (b-2a)\lambda=\delta\lambda$. Hence, $K^2\leq \frac{5}{2}\chi - 11$  by the definition of $\lambda$. Since it is clear that $(K^2, \chi)$ is a pair of strictly positive integers satisfying Noether's inequality, we conclude that $(K^2, \chi)$ satisfies the hypothesis of Theorem \ref{CorollaryFirstFam} by Remark \ref{ToArea} and there exists a canonical model $X\in\mathfrak{M}_{K^2,\chi}$ whose minimal resolution has maximal Picard number. Furthermore:
    \begin{equation*}
        \frac{K^2_X}{\chi(\mathcal{O}_X)}=\frac{(a+2b)\lambda}{b\lambda}=\frac{a}{b}+2=q-2+2=q.
    \end{equation*}
    The initial claim of this remark is now clear.
\end{Remark}

\subsection{A second family of surfaces with maximal Picard number}

The second family of surfaces of general type with maximal Picard number is described in the following:

\begin{Theorem}\label{SecondFam}
	Let $n,m,k$ be non-negative integers such that $k\leq m\leq 2n+k\neq 0$ and $k\equiv m (2)$
	and denote $K^2=10n+2m+5k-4, \chi =5n+m+2k+1$. Then $\mathfrak{M}_{K^2,\chi}$ contains canonical models  whose minimal resolution has maximal Picard number.
\end{Theorem}
\begin{proof}
	Let $ C \in|\mathcal{O}_{\mathbb{P}^2}(2)|$ be a smooth conic and $ l_1, l_2, l_3 \in|\mathcal{O}_{\mathbb{P}^2}(1)|$ be three lines such that (see \cite[Appendix A.A6]{MR0463157} for the notation):
	\begin{equation*}
		Cl_1 = 2P_1,\qquad Cl_2 = 2P_2,\qquad Cl_3 = 2P_3.
	\end{equation*}

	Let us also consider:    
	\begin{alignat*}{4}
		P_4  :=l_2 \cap l_3,\qquad  & P_5  :=l_1 \cap l_2,\qquad & l_4  :=\langle P_1,P_4\rangle,\qquad 
		& Q  := C \cap l_4\setminus \{P_1\}, \\
		l_5  :=\langle Q,P_5\rangle, \qquad & l_6  :=\langle Q,P_2\rangle, \qquad & P_6  :=l_3 \cap l_5, \qquad  & P_7  :=l_1 \cap l_6.
	\end{alignat*}
	
	\begin{center}

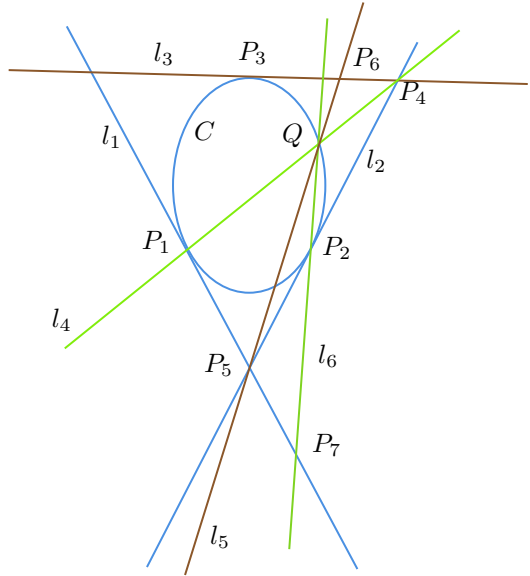
		\tikzset{every picture/.style={line width=0.75pt}} 
		
		\begin{tikzpicture}[x=0.75pt,y=0.75pt,yscale=-1,xscale=1]
			
			\draw  [color={rgb, 255:red, 74; green, 144; blue, 226 }  ,draw opacity=1 ] (272,93.5) .. controls (272,63.68) and (289.01,39.5) .. (310,39.5) .. controls (330.99,39.5) and (348,63.68) .. (348,93.5) .. controls (348,123.32) and (330.99,147.5) .. (310,147.5) .. controls (289.01,147.5) and (272,123.32) .. (272,93.5) -- cycle ;
			\draw [color={rgb, 255:red, 74; green, 144; blue, 226 }  ,draw opacity=1 ]   (219,13.5) -- (364,286.5) ;
			\draw [color={rgb, 255:red, 74; green, 144; blue, 226 }  ,draw opacity=1 ]   (394,21.5) -- (259,285.5) ;
			\draw [color={rgb, 255:red, 139; green, 87; blue, 42 }  ,draw opacity=1]   (190,35.5) -- (449,42.5) ;
			\draw [color={rgb, 255:red, 129; green, 237; blue, 12 }  ,draw opacity=1 ]   (218,175.5) -- (415,15.5) ;
			\draw [color={rgb, 255:red, 126; green, 211; blue, 33 }  ,draw opacity=1 ]   (349,9.5) -- (330,276.5) ;
			\draw [color={rgb, 255:red, 139; green, 87; blue, 42 }  ,draw opacity=1 ]   (367,1.5) -- (278,289.5) ;
			
			\draw (281,60) node [anchor=north west][inner sep=0.75pt]   [align=left] {$ C $};
			\draw (255,115) node [anchor=north west][inner sep=0.75pt]   [align=left] {$ P_1 $};
			\draw (345,118) node [anchor=north west][inner sep=0.75pt]   [align=left] {$ P_2 $};
			\draw (303,20) node [anchor=north west][inner sep=0.75pt]   [align=left] {$ P_3 $};
			\draw (383,40) node [anchor=north west][inner sep=0.75pt]   [align=left] {$ P_4 $};
			\draw (286,177) node [anchor=north west][inner sep=0.75pt]   [align=left] {$ P_5 $};
			\draw (360,23) node [anchor=north west][inner sep=0.75pt]   [align=left] {$ P_6 $};
			\draw (340,217) node [anchor=north west][inner sep=0.75pt]   [align=left] {$ P_7 $};
			\draw (235,62) node [anchor=north west][inner sep=0.75pt]   [align=left] {$ l_1 $};
			\draw (367,74) node [anchor=north west][inner sep=0.75pt]   [align=left] {$ l_2 $};
			\draw (261,20) node [anchor=north west][inner sep=0.75pt]   [align=left] {$ l_3 $};
			\draw (289,263) node [anchor=north west][inner sep=0.75pt]   [align=left] {$ l_5 $};
			\draw (210,154) node [anchor=north west][inner sep=0.75pt]   [align=left] {$ l_4 $};
			\draw (343,172) node [anchor=north west][inner sep=0.75pt]   [align=left] {$ l_6 $};
			\draw (325,61) node [anchor=north west][inner sep=0.75pt]   [align=left] {$ Q $};

		\end{tikzpicture}	    	
		
	\end{center}

	Let $\pi\colon \mathbb{F}_1 \to \mathbb{P}^2$ be the blow-up of $ \mathbb{P}^2 $ at $ Q $. We denote by:
	\begin{center}
		\begin{tabular}{l l}
			$ \overline{l}_1, \overline{l}_2, \overline{l}_3 $: & the pull-back of  $ l_1, l_2, l_3$, respectively,\\
			$ C' $, $ l'_4 $, $ l'_5, l'_6 $:&the strict transforms of $ C $, $ l_4 $, $ l_5 $, $ l_6 $, respectively.
		\end{tabular}	
	\end{center}

    Note that 
$\overline{l}_i\cong \mathcal{O}_{\mathbb{F}_1}(\Delta_0+F)$ for all $i\in\{1,2,3\}$, $l'_i\cong \mathcal{O}_{\mathbb{F}_1}(F)$ for all $i\in\{4,5,6\}$ and $C'\cong \mathcal{O}_{\mathbb{F}_1}(\Delta_0 + 2F)$.\\
    

	Let $\psi\colon \mathbb{F}_{2n+k}\to \mathbb{F}_1$ be the $ (2n+k) $-cyclic cover branched at $ l'_4 + l'_6 $ (see Lemma \ref{PerssonsTrick}).\\

    Firstly, since $ \overline{l}_1,\overline{l}_2,\overline{l}_3 \in |\mathcal{O}_{\mathbb{F}_1}(\Delta_0+F)|$  and $C' \in | \mathcal{O}_{\mathbb{F}_1}(\Delta_0 + 2F)|$ are effective divisors not having $l'_4 $ and $ l'_6$ as components, by Lemma \ref{PerssonsTrick} we obtain that    
    $$\psi^{*}\left( \overline{l}_i\right) \in | \mathcal{O}_{\mathbb{F}_{2n+k}}(\Delta_0 + (2n+k)F)| \text{ for all $i\in\{1,2,3\}$ and } \psi^{*}\left( C'\right) \in |\mathcal{O}_{\mathbb{F}_{2n+k}}(\Delta_0 + 2(2n+k)F)|.$$
    
    Set $\tilde{l}_i :=\psi^{*}\left( \overline{l}_i\right) \in | \mathcal{O}_{\mathbb{F}_{2n+k}}(\Delta_0 + (2n+k)F)|$ for all $i = 1,2,3$ and $\tilde{C} := \psi^{*}\left( C'\right) \in |\mathcal{O}_{\mathbb{F}_{2n+k}}(\Delta_0 + 2(2n+k)F)|$.\\

    Secondly, being $\psi$ branched at $ l'_4 + l'_6 $, we obtain that    
    \begin{equation*}
		\begin{split}			
			\psi^{*}\left( l'_4\right) & =(2n+k)\tilde{l}_4,\\
			\psi^{*}\left( l'_6\right) & =(2n+k)\tilde{l}_6,
		\end{split}
	\end{equation*}
    
    for some fibers $\tilde{l}_4, \tilde{l}_6 \in |\mathcal{O}_{\mathbb{F}_{2n+k}}(F)|$.\\

    Lastly, because the divisor $ l'_5 $ does not meet the branch locus $ l'_4 + l'_6$ of $\psi$, we get
    \begin{equation*}
		\begin{split}			
			\psi^{*}\left( l'_5\right) & 
			=\tilde{l}^{1}_5+\tilde{l}^{2}_5+\cdots + \tilde{l}^{2n+k}_5,
            \end{split}
	\end{equation*}
    
    for some fibers $\tilde{l}^{1}_5,\tilde{l}^{2}_5,\cdots, \tilde{l}^{2n+k}_5 \in |\mathcal{O}_{\mathbb{F}_{2n+k}}(F)|$.\\
    
	We can consider the following divisors on $ \mathbb{F}_{2n+k} $:
	\begin{equation*}
		\begin{split}
			& B_{1} := \Delta_0\\
			& B_{2}: =\begin{cases} 
				\tilde{l}_3 \text{ if } k=0, \\
				\tilde{l}_3+ \tilde{l}^{1}_5+ \cdots + \tilde{l}^{k}_5 \text{ otherwise.}\end{cases}
			\\
			& B_{3} : = \begin{cases} 
				\tilde{C} + \tilde{l}_1+ \tilde{l}_2 + 	\tilde{l}_4+ \tilde{l}_6 \text{ if } m=k, \\
				\tilde{C} + \tilde{l}_1+ \tilde{l}_2 + 	\tilde{l}_4+ \tilde{l}_6 + \tilde{l}^{k+1}_5+  \cdots + \tilde{l}^{m}_5  \text{ otherwise.}
			\end{cases}
		\end{split}
	\end{equation*}
	Note that $B_1\in|\mathcal{O}_{\mathbb{F}_{2n+k}}(\Delta_0)|,B_2\in |\mathcal{O}_{\mathbb{F}_{2n+k}}(\Delta_0 + \left( 2n+2k\right)F)|, B_3\in |\mathcal{O}_{\mathbb{F}_{2n+k}}(3\Delta_0 + \left( 8n+m+3k+2\right)F)|$.
	We can also consider the following line bundles on $ \mathbb{F}_{2n+k} $:
	\begin{equation*}
		\begin{split}
			& L_1:= 
   \mathcal{O}_{\mathbb{F}_{2n+k}}\left(2\Delta_0+\left(5n+\frac{m+5k}{2}+1\right)F\right), \\
			& L_2:=\mathcal{O}_{\mathbb{F}_{2n+k}}\left(2\Delta_0+\left(4n+\frac{m+3k}{2}+1\right)F\right),\\
			& L_3:= \mathcal{O}_{\mathbb{F}_{2n+k}}\left(\Delta_0+(n+k)F\right).
		\end{split}
	\end{equation*}

	The building data $\{L_i,B_j\}_{i,j}$ defines a $ \mathbb{Z}^2_2 $-cover $ \varphi\colon X \to \mathbb{F}_{2n+k}$. On the one hand, taking into account Lemma \ref{PerssonsTrick}, the singular set of $B_3$ consists of:
	\begin{itemize}
		\item One singularity of type $ D_{8n+4k+2} $ coming from $ P_1 $.
		\item One singularity of type $ D_{8n+4k+2} $ coming from $ P_2 $.
		\item One singularity of type $ A_1 $ coming from $ P_7 $.
		\item $ \left( 2n+k-m\right)  $ singularities of type $ A_1 $ coming from $ P_5 $.
		\item $ m-k $ singularities of type $ D_4 $ coming from $ P_5 $.
		\item $ m-k $ singularities of type $ A_1 $ coming from the intersection of $ \tilde{l}^{k+1}_5+ \cdots + \tilde{l}^{m}_5 $ and $ \tilde{C} $.
	\end{itemize}
	The singular set of $B_2$ consists of:
	\begin{itemize}
		\item $k$ singularities of type $ A_1 $ coming from $ P_6 $.
	\end{itemize}
	The singular set of $B_2+B_3$ restricted to $ B_2\cap B_3 $ consists of:
	\begin{itemize}
		\item one singularity of type $ D_{4n+2k+2} $ coming from $ P_4 $.
		\item $k$ singularities of type $ D_{4} $ coming from $ P_5 $.
		\item $ 2n+k $ singularities of type $ A_{3} $ coming from $ P_3 $.
	\end{itemize}
	Thus, the singular set of $X$ consists of:
	\begin{itemize}
		\item $ 4 $ singularities of type $ D_{8n+4k+2} $ by Remark \ref{DoubleSing}.
		\item $ 4n+2k+2 $ singularities of type $ A_1 $ by Remark \ref{DoubleSing}.
		\item  $ 2m-2k $ singularities of type $ D_4 $ by Remark \ref{DoubleSing}.
		\item One singularity of type $ D_{2n+k+2} $ by Corollary \ref{Z22Sing}, $ii)$.
		\item $k$ singularities of type $ A_3 $ by Corollary \ref{Z22Sing}, $iii)$.
		\item $ 2n+k $ singularities of type $ A_1 $ by Corollary \ref{Z22Sing}, $i)$.
	\end{itemize}
	
	On the other hand, by Proposition \ref{Z22Formulas} we have that $K_X$ is ample because $2K_{X}$ is the pull-back via $\varphi$ of the ample divisor $\mathcal{O}_{\mathbb{F}_{2n+k}}(\Delta_0 + \left( 6n+m+3k-2\right)F)$. Moreover:
	\begin{equation*}
		\begin{split}
			K_{X}^2 & = \left( \Delta_0 + \left( 6n+m+3k-2\right)F\right)^2=10n+2m+5k-4,\\
			p_g\left( X\right) & =5n+m+2k,\\
			\chi\left( \mathcal{O}_{X}\right) & =5n+m+2k+1,\\
			q\left( X\right) &  = 0,\\
			h^{1,1}\left( X\right) & = 40n+8m+15k+14.
		\end{split}
	\end{equation*}
	We conclude that $X\in\mathfrak{M}_{K^2,\chi}$ is a canonical model
	 whose minimal resolution has maximal Picard number  $   h^{1,1}\left( X\right) = 40n+8m+15k+14$ by Lemma \ref{PicardSurjectiveMorphism}.     
\end{proof}

The following result is an immediate consequence of Theorem \ref{SecondFam}. Although it is weaker than Theorem \ref{CorollaryFirstFam}, the result is included so that we can compare Theorem \ref{FirstFam} to Theorem \ref{SecondFam} (see Remark \ref{ComparingFamilies} below).

\begin{Theorem}\label{CorollarySecondFam}
	Given an integer $k\geq 0$ let  
	$(K^2, \chi)$ be an admissible pair such that $K^2=2\chi-6+k$. If $\chi\geq 3k+25$, then $\mathfrak{M}_{K^2,\chi}$ contains canonical models whose minimal resolution has maximal Picard number.
\end{Theorem}

\begin{proof}
	Let $n$ be the biggest integer such that  $\chi-(3k+1)\geq 5n$ and $\chi-(3k+1)\equiv n (2)$. Since $\chi-(3k+1)\geq 24$, it follows that $n\geq 4$. Moreover, there exists $b\in\{0,2,4,6,8\}$ such that $5n+b = \chi-(3k+1)$. If we set $m=k+b$, the result follows from Theorem \ref{SecondFam}.
\end{proof}

\begin{Remark} \label{ExtraCases2Fam}
Similarly to what happened for Theorem \ref{CorollaryFirstFam} and Theorem \ref{FirstFam}   (see Remark \ref{ExtraCases1Fam}),	the construction considered in the proof of Theorem \ref{SecondFam} works for every admissible pair $(K^2,\chi)$ satisfying the conditions stated on Theorem \ref{CorollarySecondFam}, but it also works in some other cases. For instance:
	\begin{itemize}
		\item When $K^2=2\chi-6$, this construction also produces canonical models $X$  in $\mathfrak{M}_{K^2,\chi}$  whose minimal resolution has maximal Picard number for the following pairs:
		\begin{center}
			\begin{tabular}{ |c||c|c|c|c|c|c|c|c|c|c|c|} 
				\hline
				$K^2$ & $6$ & $10$ &$16$&$20$&$24$& $26$&$30$&$34$&$36$&$38$&$40$ \\ 
				\hline 
				$\chi$ & $6$ & $8$ &$11$&$13$&$15$&$16$&$18$&$20$&$21$&$22$&$23$\\ 
				\hline
			\end{tabular}
		\end{center}
		\item When $K^2=2\chi-5$, this construction also produces canonical models $X$  in $\mathfrak{M}_{K^2,\chi}$  whose minimal resolution has maximal Picard number for the following pairs:
		\begin{center}
			\begin{tabular}{ |c||c|c|c|c|c|c|c|c|c|c|c|c| } 
				\hline
				$K^2$ & $3$ & $13$& $17$& $23$& $27$& $31$ &$33$&$37$&$41$&$43$&$45$&$47$\\ 
				\hline 
				$\chi$ & $4$ & $9$& $11$& $14$& $16$& $18$&$19$&$21$&$23$&$24$&$25$&$26$\\ 
				\hline
			\end{tabular}
		\end{center}
	\end{itemize}     
\end{Remark}

\begin{Remark}\label{ComparingFamilies}
	Taking a look at Theorem \ref{CorollaryFirstFam} and Theorem \ref{CorollarySecondFam} one realizes that the family constructed in Theorem \ref{FirstFam} is in some sense broader than the family constructed in Theorem \ref{SecondFam}. Actually, Theorem \ref{SecondFam} is not needed to prove neither Theorem \ref{CorollaryFirstFam} nor Theorem \ref{EvenHorikawaMaxPicard}. 
	Moreover, it only contributes to Theorem \ref{OddHorikawaMaxPicard} with the case $K^2=17, \chi =11$ (see Remark \ref{ExtraCases2Fam}). 
\end{Remark}


	\section{Horikawa surfaces with maximal Picard number}\label{HorikawaSurfaces}

The aim of this section is to prove Theorem \ref{EvenHorikawaMaxPicard} and Theorem \ref{OddHorikawaMaxPicard}, i.e. our results regarding  even Horikawa surfaces with maximal Picard number and odd  Horikawa surfaces with maximal Picard number respectively.
Theorem \ref{OddHorikawaMaxPicard} will require a few lemmas, but we are already able to prove Theorem \ref{EvenHorikawaMaxPicard}.

\begin{proof}[Proof of Theorem \ref{EvenHorikawaMaxPicard}]
	Let $(K^2, \chi)$ be an admissible pair such that $K^2=2\chi-6$ and $\chi\not\equiv 0 \text{ mod } 6$. Then all the connected components of  $\mathfrak{M}_{K^2, \chi}$ contain canonical models  whose minimal resolution has maximal Picard number by \cite[Theorem 1]{MR661198}. If $(K^2, \chi)$ is an admissible pair such that $K^2=2\chi-6$ and $\chi\equiv 0 \text{ mod } 6$, then $\mathfrak{M}_{K^2, \chi}$ has a unique connected component by \cite[Theorem 3.3]{MR0424831}. Therefore, it suffices to find one example for each pair of invariants satisfying the latter conditions. The case $\chi=6$ follows from \cite[Section 2, 3)]{MR661198} or Remark \ref{ExtraCases1Fam}. The case $\chi\geq 12$ follows from Theorem \ref{CorollaryFirstFam}. 
\end{proof}

In order to prove Theorem \ref{OddHorikawaMaxPicard} we will need the following:

\begin{Lemma}\label{MP13}
	$\mathfrak{M}_{1,3}$ contains canonical models  whose minimal resolution has maximal Picard number.
\end{Lemma}

\begin{proof}
	
	Let $C$ be the cuspidal cubic $(X_1^3-X_0X_2^2=0)$. Then $(X_2=0)$ is the tangent of $C$ at the cusp $(1:0:0)$ and $(X_0=0)$ is the tangent of $C$ at the inflection point $(0:0:1)$.
	We consider  a $\mathbb{Z}_2^2$-cover $\pi\colon X\to \mathbb{P}^2$  of $\mathbb{P}^2=\text{Proj}(\mathbb{C}[X_0,X_1,X_2])$ with building data $\{L_i,B_j\}_{i,j}$ where:
	$$L_1=\mathcal{O}_{\mathbb{P}^2}(1),\qquad L_2=L_3=\mathcal{O}_{\mathbb{P}^2}(3)$$ 
	and:
	\begin{equation*}
		\begin{split}
			B_1 & = \Bigl(X_0X_2(X_1^3-X_0X_2^2)=0\Bigr),\\
			B_2 & \text{ is the tangent of $C$ at a general point $P_2$},\\
			B_3 & \text{ is the tangent of $C$ at a general point $P_3\neq P_2$.}\\
		\end{split}
	\end{equation*}
	Note that $B_1\in |\mathcal{O}_{\mathbb{P}^2}(5)|,B_2,B_3\in |\mathcal{O}_{\mathbb{P}^2}(1)|$.
	On the one hand,  the singular set of $B_1$   
	consists of:
	\begin{itemize}
		\item One singularity of type $E_7$ at $(1:0:0)$.
		\item One singularity of type $A_1$ at $(0:1:0)$.
		\item One singularity of type $A_5$ at $(0:0:1)$.
	\end{itemize}
	The singular set of $B_1+B_i, i\in\{2,3\}$ restricted to $B_1\cap B_i$ consists of:
	\begin{itemize}
		\item One singularity of type $A_3$ over $P_i$.
	\end{itemize}
	Thus, the singular set of $X$ consists of:
	\begin{itemize}
		\item Two singularities of type $E_7$ over $(1:0:0)$ by Remark \ref{DoubleSing}.
		\item Two singularities of type $A_1$ over $(0:1:0)$ by Remark \ref{DoubleSing}.
		\item Two singularities of type $A_5$ over $(0:0:1)$ by Remark \ref{DoubleSing}.
		\item One singularity of type $A_1$ over $P_i, i\in\{2,3\}$ by Corollary \ref{Z22Sing},$i)$.
	\end{itemize}
	
	On the other hand, by Proposition \ref{Z22Formulas} we have that $K_X$ is ample because $2K_X$ is the pull-back via $\pi$ of the ample line bundle $\mathcal{O}_{\mathbb{P}^2}(1)$. Moreover:
	\begin{equation*}
		K_X^2 =1,\qquad
		\chi(\mathcal{O}_X) =3,\qquad
		p_g(X)  =2,\qquad q(X)=0,\qquad h^{1,1}(X)=29.
	\end{equation*}
	
	We conclude that $X\in \mathfrak{M}_{1,3}$ is a canonical model whose minimal resolution has maximal Picard number $h^{1,1}(X)=29$ by Lemma \ref{PicardSurjectiveMorphism}.
\end{proof}

\begin{Lemma}\label{MP55}
	$\mathfrak{M}_{5,5}$ contains canonical models whose minimal resolution has maximal Picard number.
\end{Lemma}
\begin{proof}
	The result follows from \cite[Theorem 1]{MR2855811}.
\end{proof}

\begin{Definition}\label{El1Singularities}
	Let $r\colon S\to X$ be the minimal resolution of a normal surface singularity $(X,p)$
	whose exceptional divisor $r^{-1}(p)$ is an elliptic curve $C$  with self-intersection $(-1)$.
	Then  $p$ is said to be a singularity of type $El(1)$. 
\end{Definition}

\begin{Lemma}\label{MP76}
	$\mathfrak{M}_{7,6}$ contains canonical models  whose minimal resolution has maximal Picard number.
\end{Lemma}

\begin{proof}
	Let $C_0,C_1,C_2$ be three conics forming Persson's tri-conical configuration \cite[Proposition 2.1]{MR661198}. In particular (see \cite[Appendix A.A6]{MR0463157} for the notation):
	\begin{equation*}
		C_0C_1= 4P_1,\quad C_0C_2= 4P_2,\quad C_1C_2=2Q+R_1+R_2.
	\end{equation*}
	
	Let $l$ be a line tangent to both $C_1$ at $T_1$ and $C_2$ at $T_2$ and let us consider the blow-up $\pi\colon\mathbb{F}_1\to \mathbb{P}^2$ of $\mathbb{P}^2$ at $Q$. Let us denote by $\Delta_0$ the exceptional divisor of $\pi$ and by $F$ the class of a fiber of $\mathbb{F}_1$. The strict transform $C$  of $C_0+C_1+C_2+l$ via $\pi$ is a curve in $|5\Delta_0+7F|$. Let us consider a $\mathbb{Z}_2^2$-cover $\phi\colon Y\to \mathbb{F}_1$ with building data $\{L_i,B_j\}_{i,j}$ where:
	\begin{itemize}
		\item  $B_1=\Delta_0+C \in |6\Delta_0+7F|$.
		\item $B_2$  is the fiber passing through the point of intersection of the strict transforms of $C_1$ and $C_2$  over $\Delta_0$.
		\item $ B_3$  is the fiber passing through $\pi^{-1}(P_1)$ and $\pi^{-1}(P_2)$.
		\item $L_1=L_2=3\Delta_0+4F$.
		\item $L_3=F$.
	\end{itemize}
	On the one hand, the singular set of $Y$ consists of:
	\begin{itemize}
		\item Two singularities of type $A_{15}$ coming from $P_1,P_2$ by Corollary \ref{Z22Sing},$iii)$.
		\item Four singularities of type $A_3$ coming from $T_1,T_2$ by Remark \ref{DoubleSing}.
		\item Four singularities of type $A_1$ coming from $R_1,R_2$ by Remark \ref{DoubleSing}.
		\item Four singularities of type $A_1$  coming from the intersection of $l$ with $C_0$ by Remark \ref{DoubleSing}.
		\item One singularity $P$ of type $El(1)$ coming from $B_2\cap \Delta_0$ (see Definition \ref{El1Singularities}).
	\end{itemize}
	
	On the other hand, the surface $Y$ has ample canonical class because $2K_Y$ is the pull-back via $\phi$ of the ample divisor $2\Delta_0+3F$. In addition, $K^2_Y=8,\chi(\mathcal{O}_Y)=7, p_g(Y)=6, q(Y)=0$.\\
	
	If we denote by $b\colon \widetilde{\mathbb{F}}_1\to \mathbb{F}_1$  the blow-up of $\mathbb{F}_1$ at $B_2\cap \Delta_0$ with exceptional divisor $E$, there is 
	a $\mathbb{Z}_2^2$-cover $\psi\colon \widetilde{X}\to\widetilde{\mathbb{F}}_1$ with branch locus $\widetilde{B}=\widetilde{B}_1+\widetilde{B}_2+
	\widetilde{B}_3$ where 
	\begin{equation*}
		\begin{split}
			\widetilde{B}_1 & =b^*B_1-3E,\\
			\widetilde{B}_2 & =b^*B_2-E,\\
			\widetilde{B}_3 & =b^*B_3+E.
		\end{split}
	\end{equation*}
	Furthermore, the minimal resolution of $P$ is precisely the map $r\colon \widetilde{X}\to Y$ making the following diagram commutative:
	\[
	\begin{tikzcd}
		\widetilde{X} \arrow{r}{r}\arrow{d}[swap]{\psi} & Y \arrow{d}{\phi} \\
		\widetilde{\mathbb{F}}_1 \arrow{r}[swap]{b}& \mathbb{F}_1
	\end{tikzcd}
	\]

    The surface $\widetilde{X}$ is of general type,  all its singularities are  $ADE$ singularities and  $K^2_{\widetilde{X}}=7, \chi(\mathcal{O}_{\widetilde{X}})=6, p_g(\widetilde{X})=5, q(\widetilde{X})=0$. Having said that, let us denote by $s\colon X\to \widetilde{X}$ the minimal resolution of $\widetilde{X}$. Then $X$ has fifty $(-2)$-curves induced by the ADE singularities of $\widetilde{X}$, but it has one more $(-2)$-curve $N$ coming from $\Delta_0$. In addition, the fiber of $X$ induced by $B_2$ consists of $2$ smooth elliptic curves $E_1,E_2$ with self-intersection $(-1)$ intersecting transversally in a single point. The intersection matrix of $N,E_1,E_2$ is:
	\begin{equation*}
		M= \left[ {\begin{array}{ccc}
				-2 & 1 & 0\\
				1 & -1 & 1\\
				0 & 1 & -1\\
		\end{array} } \right]
	\end{equation*}
	Since $N,E_1,E_2$ are disjoint from the $(-2)$-curves induced by the ADE singularities of $\widetilde{X}$ and $\rank(M)=3$, we conclude that:
	\begin{equation*}
		\rho(X)=h^{1,1}(X)=53. \end{equation*}
	
	In particular, if $X$ were minimal, it would have maximal Picard number and  it would be the minimal resolution of its canonical model, which would belong to $\mathfrak{M}_{7,6}$.  Since $2K_{X}=s^* \psi^*
	\left( b^*(2\Delta_0+3F)-E \right)$, the minimality of $X$ will follow if we show that $b^*(2\Delta_0+3F)-E$ is nef. Let us suppose  there exist non-negative integers $\alpha,\beta,\gamma$ and an irreducible and reduced curve $D\in|b^*(\alpha\Delta_0+\beta F)-\gamma E|$ such that
	$$0>D\cdot \left( b^*(2\Delta_0+3F)-E \right)=\alpha+2\beta-\gamma.$$
	This means that $b(D)$ is an irreducible and reduced  curve on $|\alpha\Delta_0+\beta F|$ that has a point of multiplicity $\gamma>\alpha+2\beta$ at $B_2\cap \Delta_0$, which is clearly impossible. Thus, $b^*(2\Delta_0+3F)-E$ is nef and the result follows. 
\end{proof}

\begin{proof}[Proof of Theorem \ref{OddHorikawaMaxPicard}]
	The result is an immediate consequence of Theorem \ref{CorollaryFirstFam},  Lemma \ref{MP13}, Lemma \ref{MP55}, Lemma \ref{MP76}, Remark \ref{ExtraCases1Fam} and Remark \ref{ExtraCases2Fam}.
\end{proof}

\begin{Remark}
	As an attentive reader may have noticed, the paper leaves several open questions. The most evident are the following:
	\begin{itemize}
		\item Does $\mathfrak{M}_{11,8}$ contain canonical models  whose minimal resolution has maximal Picard number?
		\item Let $(K^2, \chi)$ be an admissible pair such that $K^2=2\chi-5$ and $K^2+1\in 8\cdot \mathbb{Z}$ or $(K^2,\chi)=(9,7)$. Then $\mathfrak{M}_{K^2,\chi}$ has two connected components by the results of \cite{MR0460340}.  Do all the connected components of $\mathfrak{M}_{K^2,\chi}$ contain canonical models  whose minimal resolution has maximal Picard number?		
	\end{itemize}
\end{Remark}

\begin{Acknowledgments}
The authors are grateful to Professor Margarida Mendes Lopes and Professor Jungkai Alfred Chen for their valuable feedback on this article and they thank Robert Hanson for his assistance choosing its title. The authors want to thank Professor Roberto Pignatelli and Professor Sai-Kee Yeung for useful conversations on the subject of the present paper. The authors would also like to express their gratitude to the anonymous reviewers for their thorough reading of the paper and suggestions. This research is funded by Vietnam Ministry of Education and Training (MOET) under grant number B2024-DQN-02.
\end{Acknowledgments}

\begin{paracol}{2}
\BinAddresses 
  \switchcolumn
\VicenteAddresses
\end{paracol}


\begin{thebibliography}{20}
		
		\bibitem{2014arXiv1406.2143A}
		Donu {Arapura} and Partha {Solapurkar}.
		\newblock {A new class of surfaces with maximal Picard number}.
		\newblock {\em arXiv e-prints}, page arXiv:1406.2143, June 2014.
		
		\bibitem{MR2030225}
		Wolf~P. Barth, Klaus Hulek, Chris A.~M. Peters, and Antonius Van~de Ven.
		\newblock {\em Compact complex surfaces}, volume~4 of {\em Ergebnisse der
			Mathematik und ihrer Grenzgebiete. 3. Folge. A Series of Modern Surveys in
			Mathematics}.
		\newblock Springer-Verlag, Berlin, second edition, 2004.
		
		\bibitem{MR4278662}
		Ingrid Bauer and Roberto Pignatelli.
		\newblock Rigid but not infinitesimally rigid compact complex manifolds.
		\newblock {\em Duke Math. J.}, 170(8):1757--1780, 2021.

\bibitem{MR3460339}
  Ingrid Bauer and Roberto Pignatelli,
  \emph{Product-quotient surfaces: new invariants and algorithms},
  \emph{Groups Geom. Dyn.} 10(1):319--363, 2016.
 

  
		\bibitem{MR3322784}
		Arnaud Beauville.
		\newblock Some surfaces with maximal {P}icard number.
		\newblock {\em J. \'{E}c. polytech. Math.}, 1:101--116, 2014.

  \bibitem{MR3683423}
  Donald I. Cartwright, Vincent Koziarz, and Sai-Kee Yeung,
  \emph{On the {C}artwright-{S}teger surface},
  \emph{J. Algebraic Geom.} 26(4):655--689, 2017.

		
		\bibitem{MR755236}
		F.~Catanese.
		\newblock On the moduli spaces of surfaces of general type.
		\newblock {\em J. Differential Geom.}, 19(2):483--515, 1984.
		
		\bibitem{MR879190}
		F.~Catanese.
		\newblock Automorphisms of rational double points and moduli spaces of surfaces
		of general type.
		\newblock {\em Compositio Math.}, 61(1):81--102, 1987.
		
		
		\bibitem{MR3439840}
		Andrea Causin, Margarida~Mendes Lopes, and Gian~Pietro Pirola.
		\newblock The {H}odge number {$h^{1,1}$} of irregular algebraic surfaces.
		\newblock {\em Collect. Math.}, 67(1):63--68, 2016.
		
		
		\bibitem{MR498596}
		D.~Gieseker.
		\newblock Global moduli for surfaces of general type.
		\newblock {\em Invent. Math.}, 43(3):233--282, 1977.
		
		\bibitem{MR0463157}
		Robin Hartshorne.
		\newblock {\em Algebraic geometry}.
		\newblock Springer-Verlag, New York-Heidelberg, 1977.
		\newblock Graduate Texts in Mathematics, No. 52.
		
		\bibitem{MR0424831}
		Eiji Horikawa.
		\newblock Algebraic surfaces of general type with small {$c^{2}_{1}.$}\ {I}.
		\newblock {\em Ann. of Math. (2)}, 104(2):357--387, 1976.
		
		\bibitem{MR0460340}
		Eiji Horikawa.
		\newblock Algebraic surfaces of general type with small {$c^{2}_{1}$}. {II}.
		\newblock {\em Invent. Math.}, 37(2):121--155, 1976.
		
		\bibitem{MR4626843}
		Vicente {Lorenzo}.
		\newblock {Geography of minimal surfaces of general type with
			$\mathbb{Z}_2^2$-actions and the locus of Gorenstein stable surfaces}.
		\newblock {\em Math. Nachr.}, 296(6):2503-2512,  2023.
		
		\bibitem{MR4440715}
		Vicente Lorenzo.
		\newblock {$\Bbb Z_2^2$}-actions on {H}orikawa surfaces.
		\newblock {\em Manuscripta Math.}, 168(3-4):535--547, 2022.
		
		\bibitem{MR1103912}
		Rita Pardini.
		\newblock Abelian covers of algebraic varieties.
		\newblock {\em J. Reine Angew. Math.}, 417:191--213, 1991.
		
		\bibitem{MR631426}
		Ulf Persson.
		\newblock Chern invariants of surfaces of general type.
		\newblock {\em Compositio Math.}, 43(1):3--58, 1981.
		
		\bibitem{MR661198}
		Ulf Persson.
		\newblock Horikawa surfaces with maximal {P}icard numbers.
		\newblock {\em Math. Ann.}, 259(3):287--312, 1982.
		
		\bibitem{MR2855811}
		Matthias Sch\"{u}tt.
		\newblock Quintic surfaces with maximum and other {P}icard numbers.
		\newblock {\em J. Math. Soc. Japan}, 63(4):1187--1201, 2011.
		
		\bibitem{2016arXiv161100470S}
		Partha {Solapurkar}.
		\newblock {Some New Surfaces of General Type with Maximal Picard Number}.
		\newblock {\em arXiv e-prints}, page arXiv:1611.00470, November 2016.
		
	\end{thebibliography}
\end{document}